\documentclass[a4paper,12pt]{article}
\usepackage{mathptmx}
\usepackage{amsthm,amsfonts, indentfirst, latexsym, bm,graphicx,amsmath,amssymb,mathrsfs,verbatim,hyperref,pdfpages}
\usepackage{caption}
\usepackage{subcaption}
\usepackage{color}
\usepackage[T1]{fontenc}
\usepackage{dsfont}
\usepackage{latexsym}
\usepackage{url}
\usepackage{array}
\usepackage{tikz}
\usepackage{pgfplots}
\usepackage{graphicx}
\usepackage{layout}
\usepackage[titletoc,title]{appendix}
\usepackage{bbm} 
\usepackage{algorithm,algorithmic}
\numberwithin{equation}{section}
\newtheorem{theorem}{Theorem}[section]

\newtheorem{proposition}[theorem]{Proposition}

\newtheorem{lemma}[theorem]{Lemma}
\newtheorem{remark}[theorem]{Remark}

\newtheorem{definition}[theorem]{Definition}

\newcommand{\Pl}{\mathbb P}
\newcommand{\E}{ \mathbb E}

\def\E{\mathbb{E}}
\def\Var{\mathrm{Var}}
\def\Cov{\mathrm{Cov}}
\usepackage{chngcntr}
\counterwithout{equation}{section} 
\usetikzlibrary{arrows, automata,positioning,calc,shapes,decorations.pathreplacing}
  \colorlet{greencolor}{green!50!black}
  \colorlet{textcolor}{red}
  \colorlet{tancolor}{orange!80!black}
  \colorlet{bluecolor}{blue}

\usepackage{amsmath,amsfonts,amssymb,amsthm,epsfig,epstopdf,url,array}

\definecolor{plotcolor1}{rgb}{0,0.447,0.741}
\definecolor{plotcolor2}{rgb}{0.741,0,0.447}
\definecolor{plotcolor3}{rgb}{0,0.741,0.294}
\definecolor{plotcolor4}{rgb}{0.741,0.294,0}
\definecolor{plotcoloraux}{rgb}{0.447,0.447,0.447}
\tikzset{plotstyle1/.style={color=plotcolor1,solid,line width=1.0pt}}
\tikzset{plotstyle2/.style={color=plotcolor2,densely dashed,line width=1.0pt}}
\tikzset{plotstyle3/.style={color=plotcolor3,dotted,line width=1.0pt}}
\tikzset{plotstyle4/.style={color=plotcolor4,loosely dashed,line width=1.0pt}}
\tikzset{auxlines/.style={color=plotcoloraux,solid,line width=0.5pt}}
\newcommand{\markersize}{0.7pt}
\tikzset{discretemarkers/.style={mark=*,mark options={solid},mark size=\markersize}}
\usepackage{makecell} 
\def\VR{\kern-\arraycolsep\strut\vrule &\kern-\arraycolsep}
\def\vr{\kern-\arraycolsep & \kern-\arraycolsep}

\title{Markovian Transition Counting Processes: An Alternative to Markov Modulated Poisson Processes}

\author{Azam Asanjarani\footnote{The University of Auckland.}, 
Sophie Hautphenne\footnote{The University of Melbourne.}, 
Yoni Nazarathy\footnote{The University of Queensland.}.}

\date{ }

\begin{document}
\maketitle

\begin{abstract}
Stochastic models for performance analysis, optimization and control of queues hinge on a multitude of alternatives for input point processes. In case of bursty traffic, one very popular model is the \textit{Markov Modulated Poisson Process} (MMPP), however it is not the only option. 
Here, we  introduce an alternative
that we call \textit{Markovian transition counting process} (MTCP). The latter is  a point process counting the number of transitions of a finite continuous-time Markov chain.  

For a given MTCP one can establish an MMPP with the same first and second moments of counts. In this paper, we show the other direction by establishing a duality in terms of first and second moments of counts between MTCPs and a rich class of MMPPs which we refer to as slow MMPPs (modulation is slower than the events). 
Such a duality confirms the applicability of the MTCP as an alternative to the MMPP which is superior when it comes to moment matching and finding the important measures of the inter-event process.
We illustrate the use of such equivalence
in a simple queueing example, showing that the MTCP is
a comparable and competitive model for performance analysis.

\end{abstract}

\section{Introduction}
The term ``bursty traffic'' is often used in networking, traffic engineering and service systems to describe arrival patterns where at certain times there are numerous frequent arrivals and at other times there are none.  
While to a naive observer the standard Poisson process may also appear to exhibit such behavior, when referring to ``bursty traffic'' one typically means that the process is even ``more bursty'' than the Poisson process. There have been attempts to quantify and characterize the ``level of burstiness'' of a process as in \cite{neuts1993burstiness} and later in \cite{he1997episodic}, but there is no fully agreed upon definition. Instead, one often relies on specific models such as the Markov modulated Poisson process (MMPP), \cite{fischer1993markov}. The latter is a popular ``bursty traffic'' model because it has an asymptotic variance rate greater than unity (exceeds the Poisson process) and is also believed to have a squared coefficient of variation greater than unity, \cite{asanjarani2019stationary}. MMPPs can easily be integrated within stochastic models as well as network approximation schemes, \cite{kim2011modeling}. Hence to date, they have served as an attractive model for bursty traffic.

While the MMPP is popular, one may wish to find alternative models with comparable characteristics. One general option is to search within the class of \textit{Markovian arrival processes} (MAPs), of which the MMPP is a special case. See e.g. \cite{asmussen2003applied} Ch XI. In doing so, it may be interesting to find models that agree with certain attributes of a specified MMPP. Having such an alternative set of models can allow one to (i) Compare different (yet similar) models; (ii) Use the alternative models as ensembles similar to ensemble learning in a machine-learning context; (iii) Improve network decomposition schemes (such as \cite{kim2011modeling}).  In view of (i), (ii) and (iii) we present a simple yet straightforward alternative to MMPPs:  \textit{Markovian transition counting processes} (MTCPs). Our focus is on second order equivalence between a broad class of MMPPs that we call {\em slow MMPPs} and our MTCPs. From a modelling perspective and from the perspective of (i), (ii) and (iii), we argue that MTCP models can be just as useful as MMPPs. This paper aims to illustrate that.

In general, treating MAPs as {\em stationary} often yields a useful mathematical perspective which matches scenarios when there is no known dependence on time. In describing a point process we use  $N(t)$ to denote the number of events during $[0,t]$ and further use the sequence $\{T_n\}$ to denote the sequence of inter-event times. Two notions of stationarity are useful in this respect. Roughly, a MAP is {\em time-stationary} if the distribution of the number of events within a given interval does not depend on the location of the interval; that is if $N(t_1+s)-N(t_1)$ is distributed as $N(t_2+s)-N(t_2)$ for any non-negative $t_1, t_2$ and $s$. A MAP is {\em event-stationary} if the joint distribution of $T_{k_1},\ldots,T_{k_n}$ is the same as that of $T_{k_1+\ell},\ldots,T_{k_n+\ell}$ for any integer sequence of indices $k_1, \ldots,k_n$ and any integer shift $\ell$. For a given model of a point process, one may often consider either the event-stationary or the time-stationary case. The probability laws of both cases agree in the case of the Poisson process. However, this is not true in general. 

A common way to parameterise MAPs is by considering the generator, $Q$, of an irreducible finite state CTMC and setting $Q= C + D$. Roughly speaking, the matrix $C$ determines state transition rates without event counts and the matrix $D$ determines the state transition rates associated with event counts.  Here we consider two classes of MAPs as alternative models. The Markov modulated Poisson process (MMPP)  which is a natural generalization of the Poisson process and has a diagonal matrix $D$ and a Markovian transition counting process (MTCP) which has a diagonal matrix $C$ and  the diagonal elements of $D$ are all zeros.
%
%
%

When observing their counting process, $N(\cdot)$, in the MMPP setting there is no indication of jump times  in the background CTMC; but as opposed to that, in the MTCP setting, the background CTMC jumps exactly every time $t$ when $N(t)$ is incremented. This potentially makes MTCPs more attractive. 
%
   We can consider the jump chain
of an MTCP as a standard discrete-time \textit{hidden Markov model} (HMM), where the observations are the exponential random variables of sojourn times in states, and  allows us to apply the forward-backward algorithm for MTCPs' parameter estimation. Note that considering general MMPPs as standard HMMs requires more complicated
manipulations; see for instance \cite{scott2003markov}. 
%

A further virtue of MTCPs 
 is that  having a diagonal  matrix $C$ makes certain algebraic quantities regarding the inter-event time process easier to compute (typical quantities that one needs to compute are $C^{-k}$ or $e^{Cx}$ for some integer $k>0$ or real $x>0$).
%
  One of the state-of-the-arts  fitting tools for bursty data traces is the KPC-toolbox, see\cite{casale2010kpc}. The KPC-toolbox  is based on  the \textit{Kronecker product composition} (KPC) technique to generate MAPs  with predefined moments, autocorrelations, and higher order statistics in inter-event times. This technique benefits the  fact that the  Kronecker product of $k \, (k \in \mathbb{N})$  MAPs is a MAP if at least $k-1$ of them have diagonal matrix $C$,  see \cite{casale2010trace} for more details. The accurate fitting results of this method is another motive to carry on this research on  relations between MMPPs  and MTCPs.
  

In \cite{nazarathy2008asymptotic}, the authors show that for a given MTCP, one can find an MMPP with the same first and second moments of the counting process. But, can we do the converse?
Here, we consider this converse problem and show that for a wide class of MMPPs, it is possible to find MTCPs that match them exactly in terms of the mean and variance of $N(\cdot)$ for the time-stationary case.

In an early paper \cite{asanjarani2016queueing}, we handle this problem just for the case of two-state MMPPs and here, using matrix analytic methods, we generalise the solution to a $p$-state MMPP ($p \geq 2$). 
Further, for comparison, we apply moment matching to approximate  the inter-event process of an MMPP with its associated  MTCP.
 The numerical results of comparing the performance measures of the  MMPP/M/1 queue with the approximated  MTCP/M/1 queue confirm that  MTCPs can be considered as an alternative model for MMPPs.
In doing so, we compare two MTCP- based alternatives. The first is based on a detailed method of moments fit and the second is based on our simple approximated fit. We see that both methods are comparable.

The remainder of the paper is structured as follows. In Section \ref{section MAP} we overview and summarise the MAP results used in this paper. In Section~\ref{section:MMPP-MTCP} we focus on the relationships between MTCPs and MMPPs. Specifically, we present our main results for second order equivalence. We show  numerical results in Section \ref{section:data matching} and conclude in Section \ref{section:conclusion}.
Note that in this paper, Latin letters and bold notation is used for column vectors. Vectors of probabilities are row vectors and shown with bold Greek letters.

\section{Markovian Arrival Processes }
\label{section MAP}

A Markovian arrival process (MAP) is a mathematical model based on a Markov chain, used for modelling events occurring over time. A MAP of order $p$ (MAP$_p$) is generated by a two-dimensional Markov process  $\{(N(t), X(t)); t \geq 0\}$ on the state space $\{0,1, 2, \cdots\}\times \{1,2, \cdots, p\}$. The counting process $N(\cdot)$ counts  the number of  ``events''  in $[0,t]$ with  $\Pl(N(0)=0)=1$. The phase process $X(\cdot)$ is an irreducible CTMC with state space $\{1, \ldots, p\}$, initial distribution $\boldsymbol{\eta}$  and  generator matrix $Q$. A MAP is characterised by the parameters $(\boldsymbol{\eta}, C,D)$, where the matrix  $C$ has negative diagonal elements and non-negative off-diagonal elements, and records the  phase transition rates  which are not associated with an event. The matrix $D$ has non-negative elements and 
records the phase transition rates associated with an event
(increase of $N(t)$ by 1). Moreover, we have $Q=C+D$. More details are in \cite{asmussen2003applied} (Chapter~XI) and \cite{he2014fundamentals} (Chapter~2).
 
%
Since $Q$ is assumed irreducible and finite, it has a unique stationary distribution $\boldsymbol{\pi}$ satisfying $\boldsymbol{\pi} Q = \mathbf{0}'$, $\boldsymbol{\pi} \mathbf{1} = 1$, where $\mathbf{0}'$ is a row vector of 0's and $\mathbf{1}$ is a column vector of 1's. 
A MAP with parameters $(\boldsymbol{\eta}, C,D)$ is time-stationary \index{time-stationary MAP} if $\boldsymbol{\eta}=\boldsymbol{\pi}$.
 In the time-stationary  case, we have  (see \cite{asmussen2003applied}):
\begin{align}
\label{Eq:Mean}
\E[N(t)]&= {\boldsymbol{\pi}} D \mathbf{1}\,t,\\
\label{Eq:Var}
\Var\big(N(t)\big)&=\{{\boldsymbol{\pi}}D \mathbf{1}+2\, {\boldsymbol{\pi}}D D_Q^{\sharp} D \mathbf{1}\}\,t- 2 {\boldsymbol{\pi}}D D_Q^{\sharp} D_Q^{\sharp}(t) D \mathbf{1},\\
\label{Eq:third}
\E\left[N^3(t)\right]&=  6{\boldsymbol{\pi}}D \Big(\int_0^t\int_0^u e^{Q(u-s)}D \big(\mathbf{1} {\boldsymbol{\pi}} D \mathbf{1} s + D_Q^{\sharp}(s) D \mathbf{1} \,\Big)ds\,du,
\end{align}
 where $D_Q^{\sharp}(t)$ is the \textit{transient deviation matrix}, 
\begin{equation}
\label{Eq:deviationt}
D_Q^{\sharp}(t)=\int_0^{t}(e^{Qu}-\mathbf{1}{\boldsymbol{\pi}})\, du,
\end{equation}
 and $D_Q^{\sharp}$ is the \textit{deviation matrix}\index{deviation matrix}  defined by the following formula
\begin{equation} 
\label{Eq:deviation}
D_Q^{\sharp}=\lim_{t\rightarrow\infty} D^{\sharp}_Q(t)=\int_0^{\infty}(e^{Qu}-\mathbf{1}{\boldsymbol{\pi}})\, du.
\end{equation}
 Note that in some sources such as  \cite{asmussen2003applied},  the variance formula \eqref{Eq:Var}  is presented in terms of the fundamental  matrix $Q^{-}:=(\mathbf{1}{\boldsymbol{\pi}} - Q)^{-1}$.  For a given matrix $Q$, the relation between $Q^{-}$ and its deviation matrix is  $Q^{-}=D_Q^{\sharp} +\mathbf{1}{\boldsymbol{\pi}}$, see \cite{coolen2002deviation}.  
 The deviation matrix has the following properties:  
\begin{equation}\label{eq:dev}
D_Q^{\sharp}\mathbf{1}=\mathbf{0}, \quad {\boldsymbol{\pi}}D_Q^{\sharp}= \mathbf{0}, \quad \text{and} \quad D_Q^{\sharp}Q=QD_Q^{\sharp}=(\mathbf{1}{\boldsymbol{\pi}}-I),
\end{equation}
where the last one follows from the  fact that $\displaystyle \lim_{t \to \infty}e^{Qt}=\mathbf{1}\boldsymbol{\pi} $. The relation between the deviation matrix and the transient deviation matrix with the same generator matrix $Q$ is given by:
\begin{equation}\label{Eq:D-D(t)}
D_Q^{\sharp}(t)=D_Q^{\sharp}(I-e^{Qt}).
\end{equation}

 Of further interest is the embedded discrete-time Markov chain  of jump times with irreducible stochastic matrix $P = (-C)^{-1}D$ and stationary distribution $\boldsymbol{\alpha}$, where $\boldsymbol{\alpha} P=\boldsymbol{\alpha}$ and $\boldsymbol{\alpha} \mathbf{1}=1$. 
   The MAP is event-stationary
   if there is an event at time $t=0$ and $\boldsymbol{\eta}=\boldsymbol{\alpha}$. For an event-stationary MAP, the (generic) inter-event time is phase-type distributed, $PH(\boldsymbol{\alpha}, C)$ and thus has $k$-th moment:
\begin{equation}
M_k=\E[T_n^k]=k! \boldsymbol{\alpha} (-C)^{-k}\mathbf{1}
=k! \, \frac{1}{\lambda^*}( -\boldsymbol{\pi} C) (-C)^{-k} {\mathbf 1}=(-1)^{k+1} \, k! \, \frac{1}{\lambda^*} \boldsymbol{\pi} \big(C^{-1}\big)^{k-1} {\mathbf 1},
\end{equation}
where $\lambda^*= \pi D \mathbf{1}$.
Further, correlations are given by (see \cite{horvath2010joint} or \cite{horvath2013matching})
$$
\E[T_i^k \, T_{j}^l]=  k!\, l!\, \boldsymbol{\alpha}\,(-C)^{-k} P^{j-i}\,   (-C)^{-l}\mathbf{1}, \qquad j>i,
$$
 and when  the covariance  is normalised, we have  the lag-$j$ autocorrelation function
\begin{equation}\label{Eq:rho-j}
\rho_j=\frac{\Cov(T_0, T_{j})}{\Var(T_0)}=\frac{\E[T_0 \, T_j]-M_1^2}{M_2-M_1^2}.
\end{equation}
 
 The most well known type of MAP is the \textit{Markov modulated Poisson process} (MMPP). The MMPP can be considered as an arrival process which consists of a finite
number of different Poisson processes, modulated by a Markov process. In other words,
the MMPP is a particular case of doubly stochastic Poisson processes whose arrival rate is directed by transitions of a finite-state CTMC.
 So, in the standard MAP parameterization, $D$ is a diagonal matrix with entries $\boldsymbol{\lambda} = (\lambda_1,\ldots,\lambda_p)$. In practice, MMPPs are ideal for modelling bursty traffic. Examples of bursty processes are  in queues exhibiting rush hour traffic, in call centers or traffic light intersections.

Another example of a MAP is  the \textit{Markov transition counting process }(MTCP) which is a point process counting every transition of an irreducible  finite state CTMC. One can consider MMPPs and MTCPs as two extreme examples of MAPs. When there is an event in an MMPP, it never corresponds to a transition in the phase process, while when there is an event in an MTCP, it always corresponds to a transition in the phase process.
%
%
Therefore, the matrix $C$ of an MTCP, which records the phase transitions with no event, is diagonal and the diagonal elements of matrix $D$ are all zero. For an early reference that analyses both the MMPP and the MTCP (without using these terms, as such), see \cite{rudemo1973point}.


%


 It is easy to verify that the MTCP$_p$ is a natural generalisation of a Hyperexponential-$p$ renewal process (H$_p$-renewal process). Consider an MTCP$_p$ where the sojourn time in each phase $i$ (of the background CTMC) has an exponential distribution with parameter $\lambda_i$ for $i=1, \cdots, p$. The representation of an MTCP$_p$ as a MAP$_p$ is $
\boldsymbol{\pi}=(\pi_1, \cdots, \pi_p)$, $C=\text{diag} ( -\lambda_i)$, and the matrix $D$ has non-negative elements such that  $D\mathbf{1}=(\lambda_1, \cdots, \lambda_p)' $. Therefore, the distribution of inter-event times is $f(t)=\boldsymbol{\pi} e^{Ct} D\mathbf{1}=\sum_{i=1}^p\pi_i \lambda_i e^{-\lambda_i t}$ which is the density function of a H$_p$.   So, in an MTCP, times between events are H$_p$-distributed. But, they are not necessarily independent.
We also mention that the MTCP is a special case of the \textit{Markov switched Poisson process} (MSPP). The latter is a MAP where $C$ is diagonal however $D$ does not necessarily have a zero diagonal (see \cite{he2014fundamentals}). Further analysis dealing with MMPPs and MSPPs can be found in \cite{asanjarani2019stationary}.

\section{Matching MTCPs and MMPPs}\label{section:MMPP-MTCP}
In this section, we show relations between MTCPs and MMPPs. Proposition 3.2 of  \cite{nazarathy2008asymptotic} implies that every MTCP has an associated MMPP with the same first two  moments. We present this proposition in an alternative form here:
\begin{proposition}\label{prop:yoni}\cite{nazarathy2008asymptotic}
Let  $\bar{N}(t)$ be the counting process of a time-stationary MTCP$_p$. Then there is an MMPP$_p$, with  the counting process $\tilde{N}(t)$, such that their  first and second moments are matched. That is, for all $ t  \geq 0$,
$$
\E[\tilde{N}^k(t)]=\E[\bar{N}^k(t)]\,, \qquad \text{for} \quad k=1,2\,.
$$
\end{proposition}
\begin{proof}
Assume that the matrix $D$ of the MTCP$_p$ is given by $\bar{D}=Q-\text{diag}(Q)$, where $Q$ is the  generator matrix of the background CTMC. We can construct  an MMPP$_p$ with the same background CTMC by setting  $\tilde{D}=-\text{diag}(Q)$. From  \eqref{Eq:Mean} and \eqref{Eq:Var}, if we show that $\bar{D}\mathbf{1}=\tilde{D}\mathbf{1}$  and  ${\boldsymbol{\pi}} \bar{D}={\boldsymbol{\pi}}\tilde{D}$,  these processes have the same first two moments and the proof is completed. Since $Q\mathbf{1}=0$ and ${\boldsymbol{\pi}}Q =\mathbf{0}'$ the result follows.
\end{proof}
The proof shows that in order to construct an MMPP matching the first two moments with an MTCP with the same  generator matrix $Q$, we need to set $\tilde{D}= \text{diag} (\boldsymbol {\lambda})$ and $\tilde{C}=Q-\tilde{D}$,
where $\boldsymbol{\lambda} := (\lambda_1, \cdots, \lambda_p)^\prime = -\text{diag}(Q)$. 

Now the question is can we construct an MTCP whose first two moments match those of a given MMPP?  Based on the proof of the above proposition, the answer is positive for the special case of MMPPs where $\boldsymbol{\lambda}=-\text{diag}(\tilde{Q})$, i.e. $\lambda_i = \sum_{j\neq i} \tilde{q}_{ij}$. But this is a very restricted case since it does not leave any freedom with $\lambda_i$.
We now show that  for each instance of a class of MMPPs, where $\lambda_i >~\sum_{j\neq i} \tilde{q}_{ij}$  which we call ``slow MMPPs'', there is an associated MTCP whose counting process exhibits the same first and second moments.

\begin{definition}
A  \textbf{slow Markov modulated Poisson process (slow MMPP)}\index{slow MMPP} is an MMPP where the event rate in any phase $i$ is greater than the total rate of leaving  that phase, i.e. $\lambda_i > \sum_{j\neq i} q_{ij}$.
\end{definition}

\begin{figure}[h]
\centering
\begin{subfigure}[t]{.45\textwidth}
\vspace{-3.5cm}
  \centering
  \hspace{-90pt}
 \includegraphics[scale=1]{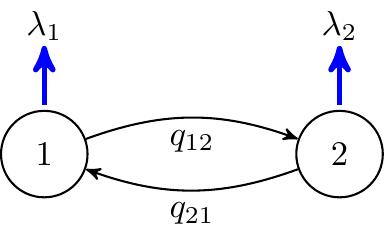}
 \vspace{1.1cm}
\caption{{ \small Transition diagram  of \\ the  phase process   of an MMPP$_2$.}} 
\end{subfigure}%
  \hspace{-50pt}
\begin{subfigure}[t]{.45\textwidth}
  \centering
 \includegraphics[scale=1]{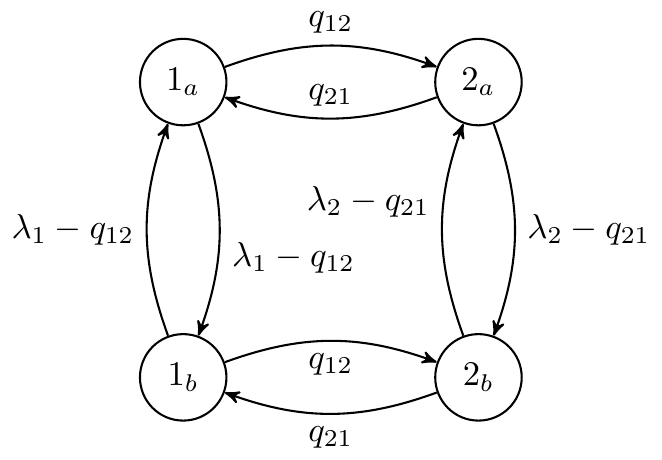}

\caption{{\small Transition diagram of the  phase process of the associated  MTCP$_4$.}} 
\end{subfigure}
\caption{{\small An MMPP$_2$ and its associated MTCP$_4$.}}
\label{Fig:Associate}
\end{figure}
We can associate an MTCP$_{2p}$ to any slow MMPP$_p$ as illustrated in Figure~\ref{Fig:Associate} for the case of $p=2$. 
As the figure shows, if the transition rate matrix  and the event intensity matrix of the slow MMPP$_2$ are given by
\begin{equation}\label{eq:Dmat}
\tilde{Q}=\left(\begin{array}{cc}
- q_{12} &  q_{12} \\
  q_{21}& -q_{21} 
\end{array}\right),
\quad
\tilde{D}=\left(\begin{array}{cc}
\lambda_1&  0 \\
  0& \lambda_2 
\end{array}\right),
\end{equation}
then, the corresponding matrices for the associated MTCP$_4$ are
\begin{equation}\label{Eq:MTCPQ}
\setlength{\arraycolsep}{0.02pt}
\bar{Q}=\bordermatrix{~&
1_a & 1_b & & 2_a & 2_b  \cr
1_a&-\lambda_1 & \lambda_1- q_{12} & \VR q_{12}& 0\cr
1_b & \lambda_1 -q_{12}&  -\lambda_1 & \VR 0 &  q_{12}\cr
\cline{2-6}
2_a &  q_{21} & 0 & \VR -\lambda_2  & \lambda_2-q_{21}\cr
2_b & 0 & q_{21}& \VR \lambda_2-q_{21} & -\lambda_2 \cr},
\quad\,\,
\bar{D}= \bar{Q}- \text{diag}(\bar{Q}).
\end{equation}
We can generalise this construction from $p=2$ to an arbitrary $p \geq 2$. Here, given a slow  MMPP$_p$, we construct an MTCP$_{2p}$ with the transition rate  and the event intensity matrices
\begin{equation}\label{Eq:generalQbar}
\bar{Q}=\left(\begin{array}{ccccc}
\bar{\Lambda}_{1} & H_{12}  & \cdots & H_{1p}\\

  H_{21} & \bar{\Lambda}_{2} & \cdots & H_{2p}\\
  \cdots & \cdots & \ddots &  \cdots\\
 H_{p1} & H_{p2}  & \cdots & \bar{\Lambda}_{p}
 \end{array}\right),
 \qquad
 \bar{D}=\left(\begin{array}{ccccc}
\bar{D}_{1} & H_{12} &  \cdots & H_{1p}\\

  H_{21} & \bar{D}_{2}  & \cdots & H_{2p}\\
 \cdots & \cdots & \ddots &  \cdots\\
  H_{p1} & H_{p2} &  \cdots & \bar{D}_{p}
 \end{array}\right),
\end{equation}
where  $\bar{\Lambda}_{i}\,$, $\bar{D}_{i}\,$, and $H_{ij}$ for $i \neq j$ and $i,j=1, \cdots, p$ are $2\times 2$ matrices given by
$$
\bar{\Lambda}_{i}=\left(\begin{array}{cc}
-\lambda_i & \lambda_{i}- S_i \\
  \lambda_{i}- S_i& -\lambda_i
 \end{array}\right),
 \qquad
\bar{D}_{i}= \left(\begin{array}{cc}
 0 & \lambda_{i}- S_i \\
  \lambda_{i}- S_i &0 
 \end{array}\right),
 \qquad
 H_{ij}=\left(\begin{array}{cc}
q_{ij}& 0 \\
 0 & q_{ij}
 \end{array}\right),
$$
where we use the notation $S_i=\sum_{j\neq i}q_{ij}$. Notice that $\bar{D}= \bar{Q}-\text{diag}(\bar{Q})$. 
 We  prove that  for an arbitrary order $p$ ($\geq 2$),  the counting processes of  the initial time-stationary slow MMPP$_p$ and its associated MTCP$_{2p}$  exhibit the same first and second moments. 
 
  In order to compare properties of the MMPP$_p$ and its associated  MTCP$_{2p}$,   we construct a MAP$_{2p}$ with the same counting process as the  MMPP$_p$. This is for  comparing processes with the same number of phases and is done by coupling the events of the phase process of the MMPP$_p$. When the process is in phase $k$, coupling events results in  a transition from phase $k_a$ to $k_b$ or vice versa. Figure~\ref{Fig:coupling} shows this for the case of $p=2$. This is nothing but a common modelling way to describe ``self- transitions'' in a CTMC. Note though that  the resulting MAP$_4$ is not an MTCP.\\ \\ \\

\begin{figure}[h]
 \centering
\begin{subfigure}[t]{.45\textwidth}
\vspace{-0.5cm}
  \centering
  \hspace{-90pt}
 \includegraphics[scale=1]{MMPP_2.pdf}
 \vspace{1.1cm}
 \subcaption{{\small Transition diagram of the phase process \\ of an MMPP$_2$.}} 
\end{subfigure}
\begin{subfigure}[t]{.45\textwidth}
\vspace{-1.5 cm}
  \centering
 \includegraphics[scale=1]{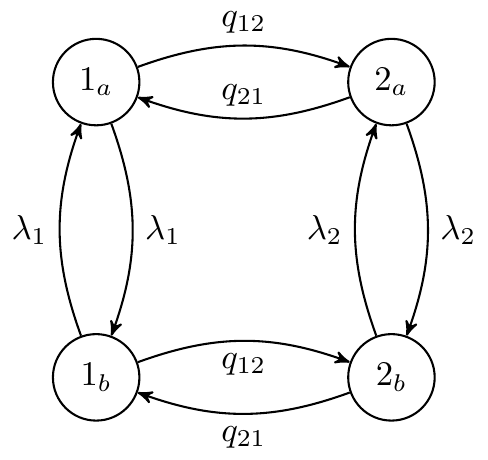}
 \subcaption{{\small Transition diagram of the phase process\\ of the coupled MAP$_4$.}} 
\end{subfigure}
\caption{{\small Construction of a MAP$_4$ from a given MMPP$_2$ by coupling.}} 
\label{Fig:coupling}
\end{figure}

We denote the phase transition matrix of the  resulting MAP$_{2p}$ by $\tilde{Q}$  and its event intensity  matrix by $\tilde{D}$. Then,   for the case of $p=2$, we have:
\begin{equation}\label{Eq:MAPQ}
\setlength{\arraycolsep}{0.02pt}
\tilde{Q}=\bordermatrix{~&
1_a & 1_b &  &2_a & 2_b  \cr
1_a &-(\lambda_1+ q_{12}) & \lambda_1& \VR q_{12}& 0\cr
1_b & \lambda_1 &  -(\lambda_1 + q_{12}) &\VR  0 & q_{12}\cr
\cline{2-6}
2_a & q_{21} & 0 &\VR -(\lambda_2 + q_{21}) & \lambda_2\cr
2_b & 0 & q_{21}& \VR \lambda_2 & -(\lambda_2 + q_{21})
\cr},
\end{equation}

\begin{equation}\label{Eq:MAPD}
 \setlength{\arraycolsep}{0.02pt}
\tilde{D}=\bordermatrix{~&
1_a & 1_b & & 2_a & 2_b  \cr
1_a &0 & \lambda_1&\VR 0& 0\cr
1_b & \lambda_1 &  0 &\VR 0 & 0\cr
\cline{2-6}
2_a & 0 & 0 &\VR 0 & \lambda_2\cr
 2_b &0 & 0&\VR \lambda_2 & 0
\cr}.
\end{equation}
Carrying out this process for an arbitrary value of $p$, the transition rate matrix and the event intensity matrix  are given by
\begin{equation}\label{Eq:generalQ}
\tilde{Q}=\left(\begin{array}{ccccc}
\tilde{\Lambda}_{1} & H_{12}  & \cdots & H_{1p}\\

  H_{21} & \tilde{\Lambda}_{2}  & \cdots & H_{2p}\\
  \cdots & \cdots & \ddots &  \cdots\\
 H_{p1} & H_{p2}  & \cdots & \tilde{\Lambda}_{p}
 \end{array}\right),
 \qquad 
\tilde{D}= \left(\begin{array}{ccccc}
\tilde{D}_{1} & 0 &  \cdots & 0\\

  0 & \tilde{D}_{2}  & \cdots & 0\\
  \cdots & \cdots & \ddots &  \cdots\\
 0 & 0 &  \cdots & \tilde{D}_{p}
 \end{array}\right),
\end{equation}
where  $\tilde{\Lambda}_{i}$ and $\tilde{D}_{i}$  are $2\times 2$ matrices given by
$$
\tilde{\Lambda}_{i}=\left(\begin{array}{cc}
-(\lambda_{i}+ S_i) & \lambda_i \\
 \lambda_i & -(\lambda_{i}+ S_i) 
 \end{array}\right),
 \qquad 
 \tilde{D}_{i}= \left(\begin{array}{cc}
 0 & \lambda_{i} \\
  \lambda_{i} &0 
 \end{array}\right).
 $$
Now, we can compare this MAP$_{2p}$ with the associated MTCP$_{2p}$.  Define  the transition rate matrix $G$ by
 $G_{ij}=q_{ij} $  for $i\neq j$,  and $ G_{ii}=-\sum_{j\neq i}q_{ij}$,
 and  denote a $p$-dimensional column vector of ones by $\mathbf{1}_p$. Then, we have the following lemma.

\begin{lemma}\label{lem:1}
  For all $ p\geq 2$:
\begin{enumerate}
\item[(i)] For any $k\geq 0$:\hspace{1cm}
$\left(I_p \otimes \mathbf{1}_2^\prime\right)\tilde{Q}^k= \left(I_p \otimes \mathbf{1}_2^\prime\right)\bar{Q}^k=G^k\left(I_p \otimes \mathbf{1}_2^\prime\right)$,

 and \hspace{2.8cm} $\tilde{Q}^k \left(I_p \otimes \mathbf{1}_2\right)= \bar{Q}^k \left(I_p \otimes \mathbf{1}_2\right)=\left(I_p \otimes \mathbf{1}_2\right)G^k$,
 
 where $\otimes$ denotes the Kronecker product.
\item[(i)$^{\prime}$] For any $t\geq 0$: \hspace{1cm}  $\left(I_p \otimes \mathbf{1}_2^\prime\right)e^{\tilde{Q}t}= \left(I_p \otimes \mathbf{1}_2^\prime\right)e^{\bar{Q}t}=e^{Gt}\left(I_p \otimes \mathbf{1}_2^\prime\right)$, 

 and \hspace{2.9cm}
$e^{\tilde{Q}t} \left(I_p \otimes \mathbf{1}_2\right)= e^{\bar{Q}t} \left(I_p \otimes \mathbf{1}_2\right)=\left(I_p \otimes \mathbf{1}_2\right)e^{Gt}$. 

\item[(ii)] Both $\tilde{Q}$ and $\bar{Q}$  have the same stationary distribution $\boldsymbol{\pi}$ which can be written as
$
\boldsymbol{\pi}=\frac{1} {2}\vartheta \left(I_p \otimes \mathbf{1}_2^\prime\right),
$
where $\vartheta$ is the stationary distribution of $G$.
\item[(iii)] ${\boldsymbol{\pi}} \tilde{D}={\boldsymbol{\pi}}\bar{D}=\frac{1}{2}\,{\vartheta}\,\text{diag}(\boldsymbol{\lambda}) \left(I_p \otimes \mathbf{1}_2^\prime\right)$, where $\boldsymbol{\lambda}=(\lambda_1, \cdots, \lambda_p)^\prime$.
\item[(iv)] $\tilde{D} \mathbf{1}_p= \bar{D} \mathbf{1}_p=\left(I_p \otimes \mathbf{1}_2\right)\boldsymbol{\lambda}$.
\item [(v)] $\left(I_p \otimes \mathbf{1}_2^\prime\right)D_{\tilde{Q}}^{\sharp}\left(I_p \otimes \mathbf{1}_2\right)= \left(I_p \otimes \mathbf{1}_2^\prime\right)D_{\bar{Q}}^{\sharp}\left(I_p \otimes \mathbf{1}_2\right)=2 D_G^{\sharp} $.
\item[(v)$^{\prime}$] $D_{\bar{Q}}^{\sharp}\left(I_p \otimes \mathbf{1}_2\right)=D_{\tilde{Q}}^{\sharp}\left(I_p \otimes \mathbf{1}_2\right)$, and 
$\left(I_p \otimes \mathbf{1}_2^\prime\right)D_{\bar{Q}}^{\sharp}=\left(I_p \otimes \mathbf{1}_2^\prime\right) D_{\tilde{Q}}^{\sharp}$.
\item[(vi)]
$D_{\bar{Q}}^{\sharp}\,\bar{D}\mathbf{1}_p=D_{\tilde{Q}}^{\sharp}\,\tilde{D}\mathbf{1}_p$, and $D_{\bar{Q}}^{\sharp}(t)\bar{D}\mathbf{1}_p=D_{\tilde{Q}}^{\sharp}(t)\tilde{D}\mathbf{1}_p$.
\item[(vii)] ${\boldsymbol{\pi}} \bar{D} D_{\bar{Q}}^{\sharp}= {\boldsymbol{\pi}} \tilde{D} D_{\tilde{Q}}^{\sharp}$.
\item[(viii)] $\left(I_p \otimes \mathbf{1}_2^\prime\right)D_{\tilde{Q}}^{\sharp}\, D_{\tilde{Q}}^{\sharp}(t)\left(I_p \otimes \mathbf{1}_2\right)= \left(I_p \otimes \mathbf{1}_2^\prime\right)D_{\bar{Q}}^{\sharp}\, D_{\bar{Q}}^{\sharp}(t)\left(I_p \otimes \mathbf{1}_2\right)=2 D_G^{\sharp}\, D_G^{\sharp}(t)$.
\end{enumerate}

\end{lemma}
\begin{proof}
The proof is in the Appendix.
\end{proof}
Having the above lemma, we can prove that there exists  an MTCP corresponding to a given slow MMPP with the same first two moments (but not necessarily the equivalent third moments). This proves the converse of Proposition~\ref{prop:yoni} for the class of slow MMPPs. 
 \begin{theorem}\label{Prop1}
Let $\tilde{N}(t)$ and $\bar{N}(t)$ be the counting processes of a time-stationary slow MMPP$_p$ and its associated MTCP$_{2p}$, respectively. Then, these processes have the same first and second moments. That is,  $\forall p \geq 2$ and $\forall t  \geq 0$,
$$
\E[\tilde{N}^k(t)]=\E[\bar{N}^k(t)]\,, \qquad  \text{for} \quad k=1,2\,.
$$
Further,  $\tilde{N}(t)$ and $\bar{N}(t)$  have different third moments. 
\end{theorem}
A proof for the case $p=2$ appears in \cite{asanjarani2016queueing} and here we present a general proof.

\begin{proof}
First, note that from \eqref{Eq:Mean} and part (iii) of Lemma~\ref{lem:1} we have 
$
\E[\tilde{N}(t)]=\E[\bar{N}(t)].
$
Then, for the variance, the proof is straightforward by using ~\eqref{Eq:Var} and parts (iii), (vi) and (vii) of  Lemma~\ref{lem:1}.
 
For the third moment, from parts (iii), (iv) and  (i)$^{\prime}$ of Lemma~\ref{lem:1}, 
we see that the first term  in the integral  of Eq. \eqref{Eq:third} is the same for both processes.
The second term in the integral  \eqref{Eq:third} for an MTCP,  can be written as
\begin{align*}
&\int_0^s {\vartheta}\,\text{diag}(\boldsymbol{\lambda}) \left(I_p \otimes \mathbf{1}_2^\prime\right) e^{\bar{Q}(u-s)} \bar{D} \big(e^{\bar{Q}v}-\mathbf{1}_p{\boldsymbol{\pi}} \big) \left(I_p \otimes \mathbf{1}_2\right)\boldsymbol{\lambda}\,dv\\
=&\int_0^s{\vartheta}\,\text{diag}(\boldsymbol{\lambda})\left( \left(I_p \otimes \mathbf{1}_2^\prime\right)e^{\bar{Q}(u-s)}\bar{D}e^{\bar{Q}v}\left(I_p \otimes \mathbf{1}_2\right)-\left(I_p \otimes \mathbf{1}_2^\prime\right)e^{\bar{Q}(u-s)} \bar{D} \mathbf{1}_p{\boldsymbol{\pi}}\left(I_p \otimes \mathbf{1}_2\right)\right)\boldsymbol{\lambda}\,dv,
\end{align*}
 where we  applied parts (iii) and (iv) of Lemma~\ref{lem:1} and the definition of the transient deviation matrix~\eqref{Eq:deviationt}. 
By using  parts (i)$^{\prime}$ and (iv) of Lemma~\ref{lem:1},  we see that second term of this integral is the same for the MMPP and its associated MTCP. 

For the first term, by using  the definition of the matrix exponential $e^{Qt}=\sum_{k=0}^{\infty}\frac{(Qt)^k}{k!}$  and putting aside the common terms or scalars, we have the following term
\begin{equation}\label{Eq:non-equal}
\left(I_p \otimes \mathbf{1}_2^\prime\right)\bar{Q}^k \bar{D}\,\bar{Q}^l\left(I_p \otimes \mathbf{1}_2\right), \quad \forall k,l \geq 0.
\end{equation}
Using  part (i) of  Lemma~\ref{lem:1}, the above term can be written as
$
G^k\left(I_p \otimes \mathbf{1}_2^\prime\right) \bar{D} \left(I_p \otimes \mathbf{1}_2\right) G^l.
$
If we rewrite  $\bar{D}$ and $\tilde{D}$ from  \eqref{Eq:generalQbar} and \eqref{Eq:generalQ} in terms of Kronecker products, we have
\begin{equation}\label{Eq:D-D}
\bar{D}=\bar{Q}+ \left( \text{diag}(\boldsymbol{\lambda})\otimes I_2\right), \qquad \tilde{D}=\text{diag}(\boldsymbol{\lambda}) \otimes\left(\begin{array}{cc}
0& 1 \\
 1 & 0
 \end{array}\right).
\end{equation}
It's easy to verify that $
(I_p \otimes \mathbf{1}_2^\prime ) \tilde{D}  =(I_p \otimes \mathbf{1}_2^\prime)\big(\text{diag}(\boldsymbol{\lambda})\otimes I_2\big)$. Therefore, 
 \begin{equation}\label{eq:Ddiffer}
(I_p \otimes \mathbf{1}_2^\prime ) \tilde{D}  (I_p \otimes \mathbf{1}_2)=(I_p \otimes \mathbf{1}_2^\prime)\big(\text{diag}(\boldsymbol{\lambda})\otimes I_2\big)  (I_p \otimes \mathbf{1}_2). 
\end{equation}
 Comparing the  above equation with \eqref{Eq:D-D} implies that the third moments are not the same and the MTCP has the following extra term
\begin{equation}\label{eq:differ}
 \int_0^s {\vartheta}\,\text{diag}(\boldsymbol{\lambda})G^k(I_p \otimes \mathbf{1}_2^\prime) \bar{Q} \left(I_p \otimes \mathbf{1}_2\right)G^l \boldsymbol{\lambda} dv,
\end{equation}
which cannot be zero unless the process is a Poisson process (see Remark \ref{re:Poi} below).
\end{proof}

 Note that Theorem~\ref{Prop1} only holds  for  slow MMPPs. Otherwise  the construction of a MAP$_{2p}$ from an MMPP$_p$ does not hold  due to some non-positive  elements $\lambda_{i}- S_i$ in  matrices $\bar{\Lambda}_{i}$ and $\bar{D}_i$.  

 \begin{remark}\label{re:Poi}
Only for  the case of Poisson process, the third moments of the slow MMPP$_p$ and its associated  MTCP$_{2p}$ are the same. This result comes from the fact that  \eqref{eq:differ} is equal to zero just for  Poisson processes:
\begin{align*}
\displaystyle
&\,\,\qquad \,{\vartheta}\,\text{diag}(\boldsymbol{\lambda})G^k(I_p \otimes \mathbf{1}_2^\prime) \bar{Q} \left(I_p \otimes \mathbf{1}_2\right)G^l \boldsymbol{\lambda}=0\\
&\Leftrightarrow \quad{\vartheta}\,\text{diag}(\boldsymbol{\lambda})G^{k+1+l}\boldsymbol{\lambda}=0,\,\quad \forall \, k,l\geq0,\\
&\Leftrightarrow \quad{\vartheta}\,\text{diag}(\boldsymbol{\lambda})G^n \boldsymbol{\lambda}=0,\,\quad \forall \, n\geq 1,
\end{align*}
where the first step holds by using part (i) of Lemma~\ref{lem:1} and the fact that $\left(I_p \otimes \mathbf{1}_2^\prime\right)\left(I_p \otimes \mathbf{1}_2\right)=2 I_p$. Since ${\vartheta}$ is the stationary distribution of $G$, the last equality holds when $\boldsymbol{\lambda}= \lambda\mathbf{1}_p$ or 
${\vartheta}\,\text{diag}(\boldsymbol{\lambda})=\lambda \,{\vartheta}$ (which implies again $\boldsymbol{\lambda}= \lambda\mathbf{1}_p$) which is the case for the Poisson process. For more details on when a general MAP is Poisson, see \cite{bean2000map}.
 \end{remark}

 Further, for any initial distribution (not restricting to the time-stationary case), we have the following proposition.
  \begin{proposition}
Let $\tilde{N}(t)$ and $\bar{N}(t)$ be the counting processes of a slow MMPP$_p$ and its associated MTCP$_{2p}$, respectively. Then, these processes have the same first moment but not necessarily the same second moment. 
\end{proposition}
\begin{proof} 
First consider the generating function of counting process of a MAP which is given by (see Chapter~XI of \cite{asmussen2003applied})
$$
F_{ij}(s,t)=\sum _{n\geq 0} s^n \Pl\Big( N(t)=n, X(t)=j|X(0)=i\Big).
$$
Moreover,  $F(1,t)=e^{Qt}$  and $F(s,t)=e^{(C+Ds)t}$. Therefore, we have
\begin{equation}\label{Eq:moment}
\frac{\partial}{\partial t}F(s,t)=(C+Ds)F(s,t).
\end{equation}
If we put $M_{ij}(t)=\E\left[ N(t)\, \mathbbm{1}_{X(t)=j}| X(0)=i\right]=\frac{\partial}{\partial s}F_{ij}(s,t)\vert_{s=1}$, then, $M(t)=\frac{\partial}{\partial s}F(s,t)\vert_{s=1}$ and
\begin{equation}\label{Eq:M}
\E_{\boldsymbol{\eta}}\left[ N(t) \right]={\boldsymbol{\eta}}M(t)\mathbf{1}.
\end{equation}
From \eqref{Eq:moment} and $M(t)=\frac{\partial}{\partial s}F(s,t)\vert_{s=1}$, we have
$
\frac{d}{d t} M(t)=De^{Qt}+QM(t), {M(0)=0}.
$
The solution of this differential equation is
$$
M(t)= \int_0^te^{Qu}De^{Q(t-u)}du.
$$
So, from \eqref{Eq:M}, the first moment of a non-stationary MAP  in terms of the transient deviation matrix is given by
\begin{equation}\label{eq:firstmom}
\E_{\boldsymbol{\eta}}\left[N(t) \right]= \boldsymbol{\eta} M(t)\mathbf{1}={\boldsymbol{\eta}} D^{\sharp}_Q(t)D\mathbf{1}+{\boldsymbol{\pi}}D\mathbf{1}t.
\end{equation}
 By using parts (iii) and (vi) of Lemma~\ref{lem:1}, we see that the first moments of both MTCP and MMPP counting processes in the non-stationary case are the same. \\ 
To show that the second moments are different, we show that the $y$-intercepts of their asymptotic variance are different.
For a non-stationary MAP, $\Var(N(t))$ has a linear $y$-intercept given by (see  \cite{hautphenne2013second})
$$
\bar{b}_{{\boldsymbol{\eta}} }=-2{\boldsymbol{\pi}}D D_Q^{\sharp}D_Q^{\sharp}D\mathbf{1}_p-2{\boldsymbol{\pi}}D\mathbf{1}_p{\boldsymbol{\eta}}\left( D_Q^{\sharp}\right)^2 D \mathbf{1}_p-\left( {\boldsymbol{\eta}} D_Q^{\sharp}D  \mathbf{1}_p \right)^2 + 2{\boldsymbol{\eta}}D_Q^{\sharp}D D_Q^{\sharp}D \mathbf{1}_p.
$$
 From parts (v$^{\prime}$), (vi) and (vii) of Lemma~\ref{lem:1}, we see that the first three terms of the above $y$-intercept are the same for an MMPP and its associated MTCP. For the last term, let's  consider an initial distribution in the form of ${\boldsymbol{\eta}}=(1, 0, \cdots, 0)\left(I_p \otimes \mathbf{1}_2^\prime\right)$. Then,  by using the definition of deviation matrix \eqref{Eq:deviation} and part (i$^{\prime}$)  of Lemma~\ref{lem:1}, we see that this term for both processes equals  to
 $$
 \int_0^{\infty}(2, 0, \cdots, 0)e^{Gt}\left(I_p \otimes \mathbf{1}_2^\prime\right)    D  D_Q^{\sharp}D \mathbf{1}_p \,dt - \int_0^{\infty}(2, 0, \cdots, 0)\left(I_p \otimes \mathbf{1}_2^\prime\right)\mathbf{1}_p \vartheta D_Q^{\sharp}D \mathbf{1}_p\,dt.
  $$
Note that from part (vi) of  Lemma~\ref{lem:1} we have $D_Q^{\sharp}D \mathbf{1}_p$ is the same for both processes and thus the second integral is the same for both processes. However, different event matrix $D$ for slow MMPPs and their corresponding MTCPs results in different left integral.  
 The explicit calculation shows that the left term in $D_Q^{\sharp}D \mathbf{1}_p$ is $\left(I_p \otimes \mathbf{1}_2\right)$, so from \eqref{Eq:D-D} and \eqref{eq:Ddiffer},  we can conclude that this  term is different for a slow MMPP and its associated MTCP. 
\end{proof}
Note that  in the proof of the last part of the above proposition, for having the same second moment, we need that the initial distribution either satisfies  ${\boldsymbol{\eta}}e^{\bar{Q}t}=0$ (which is the case for the time-stationary distribution) or $\left(I_p \otimes \mathbf{1}_2^\prime\right)\bar{Q} \left(I_p \otimes \mathbf{1}_2\right)=0$ which by considering the structure of $\bar{Q}$, results in  $ \forall i,j$: $q_{ij}=0$. Since the latter is impossible, we conclude that for having the same second moment, we must have $\boldsymbol{\eta}=\boldsymbol{\pi}$.

\section{Illustration on Matching the Inter-Event Time Process}
\label{section:data matching}

We now present a numerical illustration of second order equivalence, focusing on comparing MMPPs and MTCPs for the purpuse of modelling exogenous processes in stochastic models. We take the viewpoint of considering a queue with an MMPP arrival process and approximate it by a queue with an MTCP arrival process. 
We denote $c^2$ as the squared coefficient of variation of the event-stationary inter-event time. Further we denote $d^2$ as the limiting index of dispersion of counts, see \cite{gusella1991characterizing}. Namely:

\begin{equation}
\label{eq:3535}
c^2 = \frac{\Var(T_n)}{\E^2\,[ T_n]}, \qquad d^2 := \lim_{t \to \infty} \frac{\Var(N(t))}{\E[N(t)]}.
\end{equation} 

As a base consider an MMPP$_2$ with parameters $\tilde{q}_{12}=\tilde{q}_{21}=5$, $\lambda_1=10$, $\lambda_2=40$. Applying the formulas in Section~\ref{section MAP} for matrices in \eqref{eq:Dmat} results in 
$M_1=0.04$, $c^2=1.6923$, $M_3=0.0008 $, $\varrho_1=0.1259      $, $\varrho_2=0.0775   $, and $\varrho_3=0.0477\,$.
Say that we now wish to approximate this process by an  MTCP$_4$. How do we choose the parameters of this MTCP$_4$? For this we compare two alternatives. In one alternative we'll seek an MTCP$_4$ that closely resembles the MMPP$_2$ via moments and lag-j autocorrelations. Such a method was put forward in \cite{casale2010trace}. We refer to this alternative as the {\em optimisation method}. The other alternative we consider is to use the approximation method based on Theorem~\ref{Prop1}. We refer to this alternative as {\em second order equivalence} (SOE).

As we show in the comparative results below, the optimization method is only slightly better than SOE with a negligible difference for typical applications. However, keep in mind that carrying out an approximation via our SOE is a straight forward task, while the optimization method requires heavier computational machinery.

For the optimization method we calculate the first 50 autocorrelations of the MMPP$_2$, $\varrho_j$, from \eqref{Eq:rho-j} and minimise the following function
 $$f(\mathbf{x})=\sum_{k=1}^{50}\big(\varrho_j-\hat{\varrho}_k(\mathbf{x})\big)^2,$$
  where $\hat{\varrho}_j(\mathbf{x})$ is the lag-j autocorrelation of  an MTCP$_4$ specified by $\mathbf{x}=(x_1, \cdots, x_{12})$ and having matrices
 \begin{equation}\label{Dx}
  D=\left(\begin{array}{cccc}
 0 & x_1 & x_2 & x_3\\
  x_4&  0 & x_5& x_6\\
  x_7& x_8 &  0 & x_9\\
   x_{10} &   x_{11}&  x_{12} & 0
 \end{array}\right),
 \qquad 
C=\text{diag}(-D\mathbf{1}_4).
 \end{equation}


 Then,  applying the constrained nonlinear minimisation of the ``fmincon'' solver in MATLAB, we find the vector $\mathbf{x}$ that minimizes $f(\mathbf{x})$ with tolerance $10^{-6}$. This result is subject to:
$$\hat{M}_1(\mathbf{x}) = M_1,  \quad \text{and} \quad \hat{c}^2 (\mathbf{x}), \hat{M}_3(\mathbf{x}), \hat{\varrho}_1(\mathbf{x}), \hat{\varrho}_2(\mathbf{x}), \hat{\varrho}_3(\mathbf{x})  
$$
being constrained to their corresponding values for the MMPP$_2$, 
${c}^2, M_3, {\varrho}_1, \varrho_2, \varrho_3,
$
with a tolerance of $10^{-6}$.  

Since this is generally a non-convex optimization problem, we experiment with randomly selected initial values for the ``fmincon'' solver and then seek the optimal one. 

Using the parameters of the resulting MTCP$_4$, the matrices $D$ for both the optimization method and SOE  are respectively given by
{\normalsize
$$
 D_{\mbox{opt}}=\left(\begin{array}{cccc}
 0 & 12.3519 & 1.0078 & 0.9110\\
  9.4306&  0 & 1.2491 & 3.4953\\
  3.3339 & 2.9924 &  0 & 39.2952\\
   3.4383  &   3.5185 &  39.4102 & 0
 \end{array}\right)
 \qquad \text{and} \qquad
 D_{\mbox{{\footnotesize SOE}}}=\left(\begin{array}{cccc}
 0 & 5 & 5 & 0\\
  5&  0 & 0 & 5\\
  5& 0 &  0 & 35\\
   0  &   5 &  35 & 0
 \end{array}\right).
$$
}
Note that as opposed to $D_{\mbox{opt}}$,  the matrix $D_{\mbox{{\footnotesize SOE}}}$ is obtained immediately from Eq.~\eqref{Eq:MTCPQ}.
%
%
In Table~\ref{Tab:comp}, we present  a comparison of  the moments and first three autocorrelation values for an MMPP$_2$ and its corresponding MTCP$_4$. The  MTCP$_4$ is obtained through both the optimisation method (based on the inter-event process presented in \cite{casale2010trace}), and
 our SOE method.
As the table shows, the corresponding MTCP$_4$ obtained from  our SOE has the same first moment as the MTCP$_4$ resulting from the optimisation procedure. This comes from the fact that the first moments for both time-stationary counting process  and event-stationary inter-event process  are the same ($
M_1=\frac{1}{\boldsymbol{\pi} D \mathbf{1}_4}= \boldsymbol{\alpha} (-C)^{-1} \mathbf{1}_4
$).

 \begin{table}[h] 
 \centering
      \caption{{{\small Comparison of moments and autocorrelations of an MMPP$_2$ with associated MTCP$_4$. The parameters of the MMPP$_2$ are $q_{12}=q_{21}=5$, $\lambda_1=10$, and $\lambda_2=40$.}} } 
        \begin{tabular}{| c | c | c | c | c | c | c | c |}
    \hline
    Model & $M_1$ & $d^2$ &$c^2$ & M$_3$  & $\varrho_1$ & $\varrho_2$ & $\varrho_3$ \\
    \hline
       MMPP$_2$ ({\footnotesize  taken as ground truth})& 0.0400 & 2.8000& 1.6923& 0.0008 & 0.1259    &0.0775  &  0.0477\\
    \hline
     MTCP$_4$ ({\footnotesize  optimisation method}) & 0.0400 &2.7933& 1.6923& 0.0008 &  0.1256  &  0.0771&    0.0473 \\
     \hline
     MTCP$_4$ ({\footnotesize  SOE}) & 0.0400 & 2.8000& 2.1250& 0.0013 &   0.0993   & 0.0372 &    0.0140 \\
     \hline
    \end{tabular}  
    \label{Tab:comp}
 \end{table} 

As we expect, the value of $d^2$ which is related to the counting processes are exactly the same for the given MMPP$_2$ and the corresponding MTCP$_4$ that resulted from our SOE method. Based on Table~\ref{Tab:comp},
 the MMPP$_2$ and the MTCP$_4$ resulting from optimisation method are of comparable moments and autocorrelations (their differences are less than $10^{-3}$). But, as expected, this is not the case for  the MTCP$_4$ obtained from our SOE method. 

 \begin{figure}[h!]
  \centering
  \begin{subfigure}{0.65\textwidth}
    \includegraphics[scale=0.5]{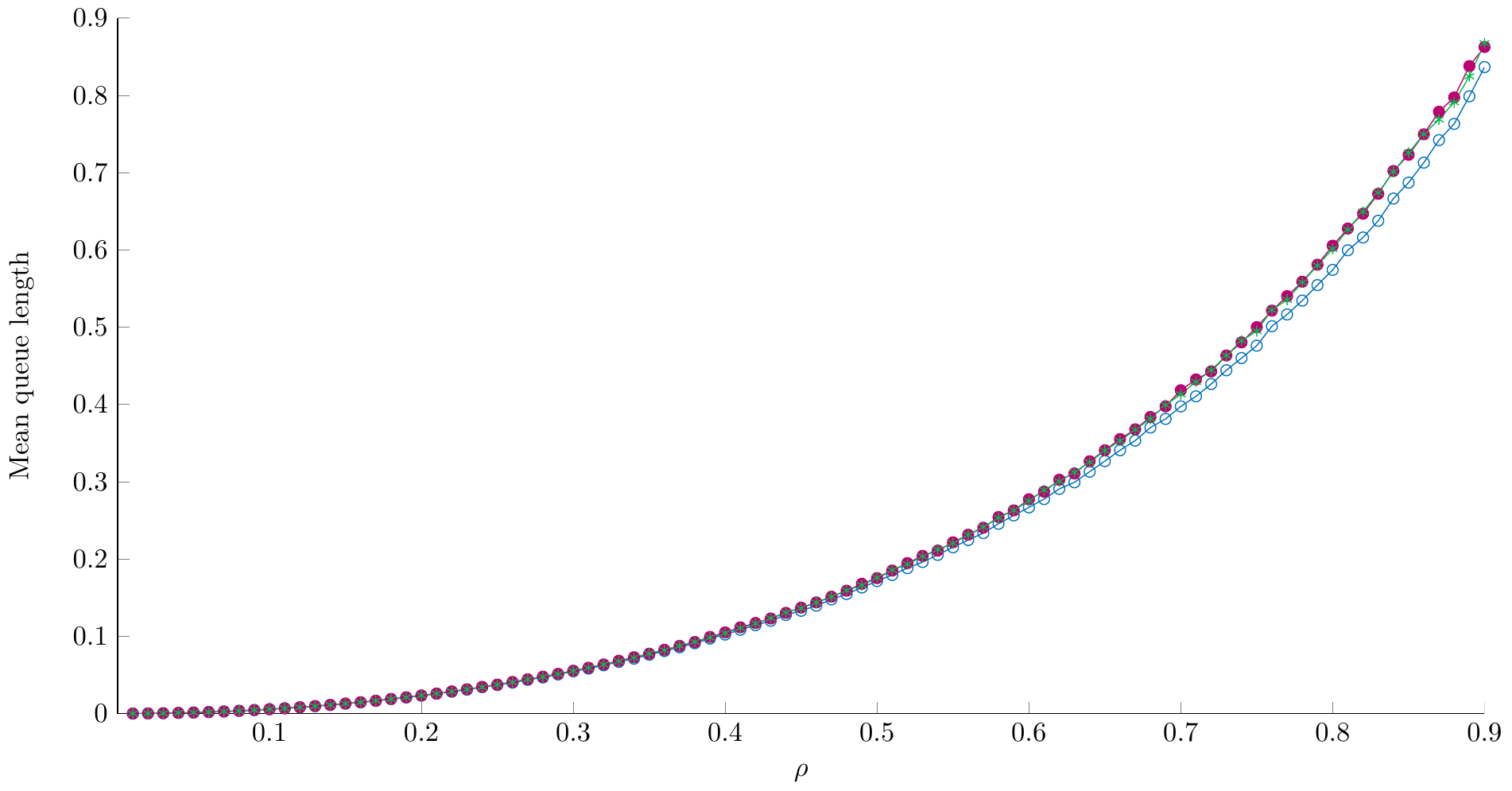}
    \includegraphics[]{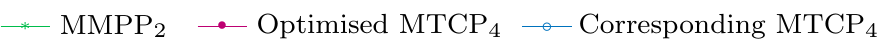}
    \caption{Comparison of mean queue length.}
    \label{Fig:meanqueue}
  \end{subfigure}
   \begin{subfigure}{0.65\textwidth}
    \includegraphics[scale=0.5]{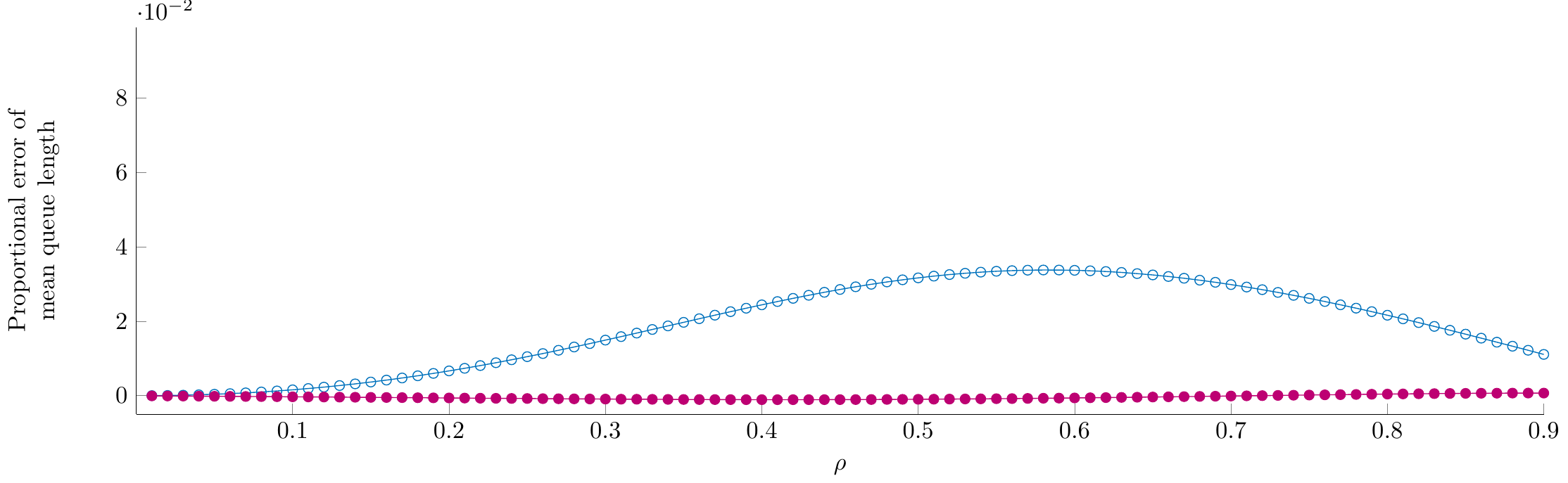}
\includegraphics[]{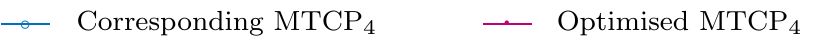}
    \caption{Comparison of proportional error of mean queue length.}
\label{Fig:PropErr}
  \end{subfigure}
  \begin{subfigure}{0.65\textwidth}
   \includegraphics[scale=0.5]{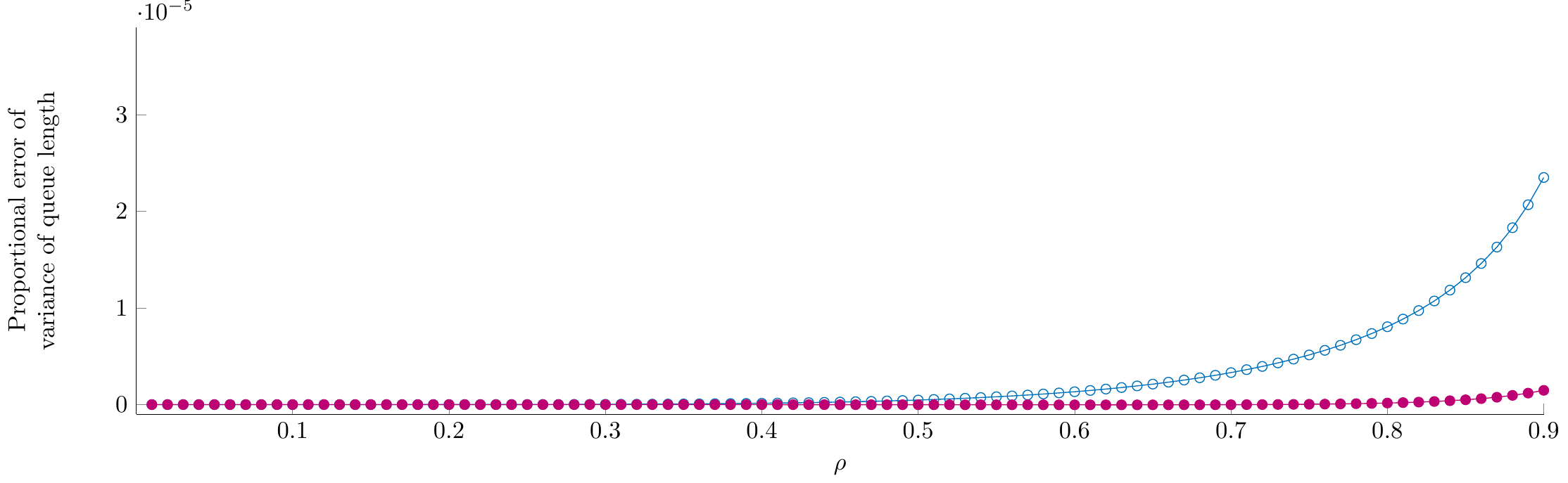}
\includegraphics[]{capErr.pdf}
    \caption{Comparison of proportional error of variance of queue length.}
   \label{Fig:PropErrVar}
  \end{subfigure}
  \caption{Performance comparison for two alterantive MTCP$_4$/M/1 approximations of the queueing system. The optimization method and the SOE (corresponding MTCP). Here, the workload varies from 0 to 0.9.}
\end{figure}

To see the effect on an example stochastic model, we now consider a queueing system with  matrix analytic analysis. Assume that we have an  MMPP$_2$/M/1 queue and we wish to approximate it by an MTCP$_4$/M/1 queue. Consider the MTCP obtained through the optimization method and through the SOE method and compare how each of these arrival process approximations behave in terms of queue performance analysis. 

Figure~\ref{Fig:meanqueue} demonstrates differences between the mean queue length curves of the 
   MMPP$_2$ /M/1 queue (green curve) and the  MTCP$_4$/M/1 queue. The MTCP$_4$ obtained either  from the optimisation method (red curve) or from our SOE (blue curve). Note that the value of workload $\rho$ varies from $0$ to $0.9$ by changing the  mean of the service time. 
 Furthermore, the resulting mean queue lengths in Figure~\ref{Fig:meanqueue} are calculated by considering the MAP/M/1 queue as a QBD and then using the SMC solver, see \cite{asanjarani2016queueing}  for more details. The same method is used to obtain the corresponding measures for Figure~\ref{Fig:PropErr} and  Figure~\ref{Fig:PropErrVar}.
 
Figure~\ref{Fig:PropErr} presents a comparison of the proportional error of mean queue length of a steady-state  MMPP$_2$/M/1 queue where the MMPP$_2$ is approximated by an MTCP$_4$. This MTCP$_4$ is obtained from either the optimization method (red curve) or our proposed SOE  (blue curve). As the figure shows, this error is zero and by increasing the value of $\rho$ increases to at most $3.5\times 10^{-2}$ and then again decreases.
Hence, while the MTCP$_4$ resulting from optimisation procedure results in closer results to the original model, the proportional error of approximating with the MTCP$_4$ resulting from the (simpler) SOE is  negligible too.  
To reinforce this observation, we consider  the proportional error in measuring another important performance measure, the variance of queue length where approximating an MMPP$_2$/M/1 queue with an MTCP$_4$/M/1 queue. Figure~\ref{Fig:PropErrVar} demonstrates the proportional error of variance of queue length when approximating an  MMPP$_2$/M/1 queue with an MTCP$_4$/M/1 queue. The blue curve is obtained by considering the SOE in Section~ \ref{section:MMPP-MTCP} and the red curve is obtained by considering the optimisation procedure for finding the corresponding MTCP$_4$. In both cases by changing the service rate, the value of workload $\rho$ varies from $0$ to $0.9$. As  Figure~\ref{Fig:PropErrVar} shows,  the proportional error of variance with our SOE is less than $3 \times 10^{-5}$.

Therefore, regarding the simplicity of finding the associated MTCP$_4$ suggested in Section~ \ref{section:MMPP-MTCP} and negligibility of the proportional error of mean queue and variance of queue length, we can suggest that our method is suitable for approximating an MMPP$_2$ with an MTCP$_4$ in queueing applications.


\section{Conclusion and Open Problems}
\label{section:conclusion}
 For the two extreme examples of MAPs, the MMPP (where there is no event at transitions of the underlying CTMC) and the MTCP (where all transitions of the underlying CTMC are accompanied with an event), we have shown that the first two moments of counts for a large class of MMPPs (slow MMPPs) and MTCPs are the same. Therefore, from a modelling point of view, one can construct an MMPP from a given MTCP, and for a given slow MMPP  there is an MTCP with the same first moments of counts. A question that arises is whether there is any such similar relation between the class of MMPPs and the class of MTCPs? In other words, can we find a relation between non-slow MMPPs and MTCPs (or maybe another class of MAPs)? 
 
 We also carried out parameter estimation for these MAPs. We consider parameters of a slow MMPP$_2$ and use the  nonlinear optimisation procedure presented in \cite{casale2010trace} to  estimate  parameters of an MTCP$_4$ with the same first moments and  autocorrelations of the inter-event time process. Comparing the results of this method with  parameters resulting from matching the first two moments of counts in the context of a queueing model in \cite{asanjarani2016queueing} shows that the differences between these results are negligible. There are some papers related to fitting data with a MAP, see for instance \cite{okamura2016fitting} and references therein. But, as time-homogeneity of the data trace is a significant property in available fitting methods, a question that arises here for future work is  what is the best fitting method if the real data trace is non-homogeneous?



{\Large \textbf{{Appendix}}}\\

\textbf{Proof of Lemma~\ref{lem:1}}
The proof of (i) follows from the mathematical induction and the  structure of  matrices in \eqref{Eq:generalQbar} and \eqref{Eq:generalQ}. Note that for $k=0$, the relations are obvious and for $k=1$, it is easy to check that multiplying both $\tilde{ Q}$ and $\bar{Q}$ from the left by $I_p \otimes \mathbf{1}_2^\prime$ gives the same result. Further, for instance, if we consider that $(I_p \otimes \mathbf{1}_2^\prime)\tilde{Q}^{k}=G^k (I_p \otimes \mathbf{1}_2^\prime)$, then, for $k+1$ we have:
\begin{align*}
(I_p \otimes \mathbf{1}_2^\prime)\tilde{Q}^{k+1}&=\Big( \left(I_p \otimes \mathbf{1}_2^\prime\right)\tilde{Q}^k\Big) \tilde{Q}\\
\text{(by induction assumption)}\quad&=\Big( G^k(I_p \otimes \mathbf{1}_2^\prime)\Big)\tilde{Q}= G^k \Big( (I_p \otimes \mathbf{1}_2^\prime)\tilde{Q}\Big)\\
\text{(by induction assumption)}\quad &= G^k \Big(G (I_p \otimes \mathbf{1}_2^\prime)\Big)= G^{k+1} (I_p \otimes \mathbf{1}_2^\prime).
\end{align*}
The other relations can be proved similarly.\\
 The proof of (i)$^{\prime}$ is a consequence of (i) and  the definition of the matrix exponential $e^{Qt}=\sum_{k=0}^{\infty}\frac{(Qt)^k}{k!}$. 
 \\
For part (ii), we need to show that ${\boldsymbol{\pi}}{\tilde{Q}}= {\boldsymbol{\pi}}\bar{Q}=0$ and ${\boldsymbol{\pi}}\mathbf{1}_p= 1$. First,  from part (i), we have
$$
{\boldsymbol{\pi}} {\tilde{Q}}=\frac{1} {2}{\vartheta} \left(I_p \otimes \mathbf{1}_2^\prime\right){\tilde{Q}} =\frac{1} {2}{\vartheta}\, G \left(I_p \otimes \mathbf{1}_2^\prime\right)=0,
$$
where the last equality comes  from the fact that ${\vartheta}$ is the stationary distribution of $G$. Moreover, we have  
$$\begin{array}{lll}
{\boldsymbol{\pi}}\mathbf{1}_p= \frac{1} {2}{\vartheta} \left(I_p \otimes \mathbf{1}_2^\prime\right)\mathbf{1}_p
= \frac{1} {2}{\vartheta} \left(\begin{array}{ccc}
2\\ \vdots \\2
 \end{array}\right)
={\vartheta}\, \mathbf{1}_p
=1,
\end{array}
$$
where  again, for the last equality we use the fact that ${\vartheta}$ is a distribution.   The  proof for the case of $\bar{Q}$ follows the same lines. \\
 By considering  part (ii) and the structure of matrices $\bar{D}$ and  ${\tilde{D}}$ in  \eqref{Eq:generalQbar} and \eqref{Eq:generalQ},  proofs of (iii) and (iv) are obvious.\\
 For  part (v), we  use the definition of the deviation matrix  in \eqref{Eq:deviation} as follows
\begin{align*}
\left(I_p \otimes \mathbf{1}_2^\prime\right)D_{\tilde{Q}}^{\sharp}\left(I_p \otimes \mathbf{1}_2\right)
&=\int_0^{\infty}\left(I_p \otimes \mathbf{1}_2^\prime\right)(e^{\tilde{Q}u}-\mathbf{1}_p{\boldsymbol{\pi}})\left(I_p \otimes \mathbf{1}_2\right)\, du\\
&= \int_0^{\infty}\left(\left(I_p \otimes \mathbf{1}_2^\prime\right)e^{\tilde{Q}u}\left(I_p \otimes \mathbf{1}_2\right)-\left(I_p \otimes \mathbf{1}_2^\prime\right)  \mathbf{1}_p{\boldsymbol{\pi}}   \left(I_p \otimes \mathbf{1}_2\right)\right)\, du\\
\big(\text{by part}\, (i)^\prime \text{and part}\, (ii)\big)\quad
&= \int_0^{\infty}\left(e^{Gu}\left(I_p \otimes \mathbf{1}_2^\prime\right)\left(I_p \otimes \mathbf{1}_2\right)-\left(I_p \otimes \mathbf{1}_2^\prime\right)  \mathbf{1}_p\,\frac{1} {2}\vartheta \left(I_p \otimes \mathbf{1}_2^\prime\right)   \left(I_p \otimes \mathbf{1}_2\right)\right)\, du\\
&= 2 \int_0^{\infty}(e^{Gu}-\mathbf{1}_p{\vartheta})\, du\\
&=2 D_G^{\sharp}\,.
\end{align*}
In the proof, we used the formula $(A\otimes B)(C\otimes D)=(AC\otimes BD)$  (when matrix dimensions agree for the multiplication) to show that $
(I_p \otimes \mathbf{1}_2^\prime)(I_p \otimes \mathbf{1}_2)= (I_p \otimes 2)=2I_p.
$
Then, since  $(I_p \otimes \mathbf{1}_2^\prime) \mathbf{1}_p=2\, \mathbf{1}_p $, we have the result.\\
   The proof for the case of $\bar{Q}$ follows the same lines and the proof of (v)$^{\prime}$ is a corollary of the proof of part (v). Proofs of   (vi) and (vii) are the results of  parts (iii), (iv) and (v)$^{\prime}$.\\
   %
For  part (viii),  from   \eqref{Eq:deviation} and \eqref{Eq:deviationt},  we have: 
$$\begin{array}{lll}
\left(I_p \otimes \mathbf{1}_2^\prime\right)D_{\tilde{Q}}^{\sharp}D_{\tilde{Q}}^{\sharp}(t)\left(I_p \otimes \mathbf{1}_2\right)
= \int_0^{\infty} \int_0^{t} \left(I_p \otimes \mathbf{1}_2^\prime\right) (e^{\tilde{Q}(u+v)}- \mathbf{1}_p{\boldsymbol{\pi}})\left(I_p \otimes \mathbf{1}_2\right) \,dv\,du \\
\hspace{47.5mm}=2 \int_0^{\infty} \int_0^{t} (e^{G(u+v)}-\mathbf{1}_p{\vartheta})\,dv\,du\\
\hspace{47.5mm}=2 D_G^{\sharp} D_G^{\sharp}(t),  
\end{array}$$ 
where in addition to  previous relations, in the first step  we use the fact that for any transition rate matrix $Q$ with stationary distribution ${\boldsymbol{\pi}}$, we have
 $$e^{Qt}\mathbf{1}_p= \mathbf{1}_p \text{ (since }\,e^{Qt} \, \text{is stochastic}),\qquad{\boldsymbol{\pi}}e^{Qt}={\boldsymbol{\pi}}. $$
 Further, we have $
\mathbf{1}_p \vartheta\,\mathbf{1}_p \vartheta=\mathbf{1}_p \vartheta 
$.
The proof for $\bar{Q}$ follows  the same lines.
\end{document}